\documentclass{llncs}

\usepackage{latexcad}
\usepackage{amssymb} 
\usepackage[all]{xy} 
\usepackage{enumerate} 
\newcommand{\cat}[1]{\textbf{#1}} 

\newdir{ >}{{}*!/-10pt/@{>}}            
\newcommand\TLTS{2LTS}
\usepackage{amsmath}
\usepackage{pifont}

\begin{document}

\title{The parallel composition of processes}

\author{L. de Francesco Albasini \and N. Sabadini \and R.F.C. Walters
\thanks{The authors gratefully acknowledge financial support from 
the  Universit\'a dell'Insubria,  and the Italian Government PRIN project ART ({\em Analisi di sistemi di Riduzione mediante sistemi di Transizione)}.}}
\institute{Dipartimento di Scienze delle Cultura,\\
Politiche e dell'Informazione,\\
Universit\`a dell' Insubria, Italy}
\date{}

  \maketitle

\hyphenation{Milner} \hyphenation{ope-ra-tions} \hyphenation{diffe-rent}

\begin{abstract}
We suggest that the canonical parallel operation of processes is composition in a well-supported compact closed category of spans of reflexive graphs. We present the parallel operations of classical process algebras as derived operations arising from monoid objects in such a category, representing the fact that they are protocols based on an underlying broadcast communication.
\end{abstract}

\section{Introduction}

The algebraic structure of sequential operations on processes has been studied since the beginning of computer science, with recent important contributions being \cite{BE93,JSV96}. Parallel operations
have been studied with less success, with a variety of different
\emph{process algebras} arising and no real consensus on the basic
operations. In this note we would like to argue that the
well-supported compact closed category (wscc) of spans of reflexive
graphs introduced in \cite{KSW97} is in fact a \emph{canonical
algebra for parallel composition}. We will present
the particular algebras introduced by Milner \cite{M89}, Hoare \cite{H85} and others as derived operations of the wscc structure and additional commutative monoid objects  in Span(RGraph) (generalizing Winskell's synchronization algebras\ \cite{WN}). One reason previous authors have considered these derived
operations is their desire for a single point of observation of a process, which has been confused with the quite different idea of
\emph{interleaving} semantics. Another reason is that conventional process algebras
assume a form of broadcast communication between processes, exactly
achieved by the operations of the monoid object.

Our suggestion is coherent with remarks made by Abramsky in \cite{A} which we quote here. He makes these criticisms, namely that in traditional process algebras 

(i) \textquotedblleft interaction becomes extrinsic: we must add some additional structure,
typically a  \textquoteleft synchronization algebra' on the labels, which
implicitly refers to some external agency for matching up labels and
generating communication events, rather than  finding the meaning of
interaction in the structure we already have." (the extra monoid object)

(ii) \textquotedblleft interaction becomes ad hoc: because it is an \textquoteleft invented' additional
structure, many possibilities arise, and it is hard to identify any as
canonical" (non-canonicity) °°   

(iii) \textquotedblleft interaction becomes global: using names to match up communications
implies some large space in which potential communications \textquoteleft swim' " (broadcast).

We agree with these three points, while firmly disagreeing with Abramsky's acceptance in that paper that (\textquotedblleft extensional")  behaviours  should be described rather than (\textquotedblleft intensional") systems. It is our contention that there should be a common algebra for systems and behaviours, compositionality being the existence of a morphism (actual behaviour) between the two. Unfortunately this  is  lacking even in classical treatments of sequential processes - Kleene expressions describe behaviour but not automata. In  \cite{RSW08} the algebra of this paper applied to cospans rather than spans to provide an algebra of automata and behaviours for which so that the Kleene theorem is a statement of compositionality. In the case of process algebras there is no notion which corresponds to the systems, only the behaviour - but the behaviour of what? It will be clear from this paper that we are are firmly in favour of an automata semantics  of process algebras, rather than the often incompatible, but commonly accepted, structural operational semantics.  Of course, from any algebra of automata one may produce a rewrite system, as we have done for our algebra in \cite{KSW08} .

 The abstract algebra described in this paper was introduced in
\cite{CW87,W87} and developed in the context of computer science in \cite{KSW97a,KSW97b,KSW97,KSW00a,KSW00b,KSW02}.
The algebra has also found application in quantum field theory
\cite{K04} and quantum experiments \cite{AC04}.

We describe the contents of the paper in more detail. In section 2 we give the abstract result that a pair of objects $X$, $Y$ in a symmetric monoidal category, $X$ with a comonoid structure and $Y$ with a monoid structure induce a monoid structure on $Hom(X,Y)$. 

In section 3 we introduce a simpler variant $\TLTS$  of the category $Span(RGraph)$, whose arrows are $\emph{two-sided labelled transition systems,}$ by which we mean spans of reflexive graphs which are jointly monic on arcs - there is at most one arc between two states with a given (double) labelling. This  has  the same algebraic structure as $Span(RGraph)$, but is more convenient in making comparison with classical labelled transition systems. We describe briefly the algebraic structure.

In section 4 we note that the synchronization algebras of Winskel \cite{WN} are particular commutative semigroups in $\TLTS$, and we show that the parallel composition of transition systems in \cite{WN} is exactly our construction of section 2.     Notice that familiar examples of synchronization algebras  are merely semigroups in the setting of \cite{WN}, but considered in $\TLTS$ they are actually monoids with the results that processes under the associated parallel operations are monoids, not merely semigroups.   
 Notice that the usual interpretation that processes have many channels is seen in our analysis to be misleading - in fact processes communicate on a single channel or bus, the mediation with this bus is provided by the monoid object. This is made particularly clear by the geometry corresponding to the algebra. Naturally, on the bus there is interleaving but  by no means are the various processes acting in interleaving internally. A further technicality which has lead to the confusion between the fact that processes interact through a bus and interleaving semantics is the lack of appreciation of the reflexive graphs. In \cite{WN} it is prohibited that the synchronization of two actions results in the null action $\varepsilon$. This means that internal actions are always mirrored on the bus.

Further comments on process algebras by the second and third authors may be found at \cite{W08}.

\section{Monoid objects}

A \emph{commutative monoid object} in a symmetric strict monoidal category
\cat{M} (with tensor $\otimes$, identity of tensor $I$, and symmetry $tw$) consists of an object $M$  with arrows
\begin{displaymath}
e:I \to M,\qquad m: M \otimes M \to M
\end{displaymath}
such that the following diagrams commute:

\begin{equation}\label{3.1}
\xymatrix{M \otimes M \otimes M \ar[r]^{1 \otimes m} \ar[d]_{m \otimes 1} & M \otimes M \ar[d]^m \\
M \otimes M \ar[r]_m & M}
\end{equation}
\[\] 
\begin{equation}\label{3.2}
\xymatrix{I \otimes M \ar[r]^{e \otimes 1} \ar[dr]_{1} & M \otimes M
\ar[d]^m
&& M \otimes M  \ar[d]_m & M \otimes I \ar[l]_{1 \otimes e} \ar[dl]^{1} \\
& M && M}
\end{equation}
\[\]
\begin{equation}\label{3.3}
\xymatrix{M \otimes M \ar[rd]_{m} \ar[r]^{tw} & M \otimes M \ar[d]^m \\
& M}
\end{equation}
There is a geometrical representation of expressions and equations in a symmetric monoidal category (see, for example, \cite{KSW97b}). The arrows $m$ and $e$ are represented respectively as:

\noindent\centerline{{\tt\setlength{\unitlength}{0.60pt}
\begin{picture}(260,60)
\thinlines
\put(50,17){$m$}
\path(0,30)(40,30)(40,0)(60,0)(80,20)(100,20)
\path(0,10)(40,10)(40,40)(60,40)(80,20)(100,20)
\path(180,20)(180,33)(200,33)(200,7)(180,7)(180,20)
\path(200,20)(240,20)
\put(-20,30){$M$}
\put(-20,0){$M$}
\put(100,15){$M$}
\put(245,15){$M$}
\put(185,15){$e$}
\end{picture}}}

\bigskip
\noindent Then axiom (\ref{3.1}) becomes

\centerline{{\tt\setlength{\unitlength}{0.60pt}
\begin{picture}(260,60)
\thinlines
\put(50,17){$m$}
\put(130,5){$m$}
\put(270,-3){$m$}
\put(350,9){$m$}
\put(190,5){$=$}
\path(0,30)(40,30)(40,0)(60,0)(80,20)(100,20)
\path(0,10)(40,10)(40,40)(60,40)(80,20)(100,20)
\path(100,20)(120,20)(120,30)(140,30)(160,10)(140,-10)(120,-10)(120,20)
\path(120,-5)(0,-5)
\path(160,10)(180,10)
\path(220,-15)(260,-15)
\path(220,10)(260,10)
\path(260,0)(260,-20)(280,-20)(300,0)(280,20)(260,20)(260,0)
\path(220,25)(340,25)(340,35)(360,35)(380,15)(360,-5)(340,-5)(340,25)
\path(340,0)(300,0)
\path(380,15)(400,15)
\end{picture}}}

\vspace{30pt} 

If the arrow $e$ is omitted in the above definition we get the notion of \emph{commutative semigroup object} in $\cat{M}$. Dually, a \emph{cocommutative comonoid object} in a monoidal category
\cat{M} is a monoid $M$ in the dual category $\cat{M}^{op}$. That
is, an object $M$ with two arrows
\begin{displaymath}
\begin{array}{c}
c: M \to M \otimes M, \qquad e':M \to I \\
\end{array}
\end{displaymath}satisfying the axioms dual to the monoid axioms.  There is similarly an obvious notion of  cocommutative cosemigroup.

\begin{proposition}
If $X$ and $Y$ are respectively a cocommutative comonoid and a commutative monoid object of a
symmetric monoidal category \cat{M}, then $Hom(X,Y)$ has an induced commutative monoid structure with multiplication being

\begin{displaymath}
\begin{array}{c}
\bullet : \xymatrix@R=1pt{ Hom(X,Y) \times Hom(X,Y) \ar[r] & Hom(X,Y) \\
(A,B) \ar@{|->}[r] & f \bullet g := m (A \otimes B) c}
\\
\end{array}
\end{displaymath}
and identity 
$$\xymatrix{\mathbf{e}=X \ar[r]^{e'} & I \ar[r]^{e} & Y}.
$$ An analogous result holds for $X$ a cosemigroup, and $Y$ a semigroup.
\end{proposition}
\begin{proof}

By the associativity of m and c, $\bullet$ is associative:
\begin{eqnarray}
\lefteqn{ A \bullet (B \bullet C)= m \bigg(  A \otimes \big( m ( B
\otimes C)c
\big) \bigg)c = m (1 \otimes m ) (A \otimes B \otimes C) (1 \otimes c)c= } \nonumber \\
& & {} = m (m \otimes 1 ) (A\otimes B \otimes C) (c \otimes 1)c= m
\bigg(  \big( m (A \otimes B) c \big ) \otimes C \bigg)c = ( A\bullet B ) \bullet C. \nonumber
\end{eqnarray}
The identity axiom follows since
\begin{displaymath}
\begin{array}{c}
A \bullet \mathbf{e}= m \big( A \otimes (e \, e') \big)c = m (1
\otimes e ) (A \otimes 1) (1 \otimes e')  c=
1 \, (A \otimes 1) \, 1= A\\
\mathbf{e} \bullet A = m \big( (e \, e') \otimes A \big) c= m (e
\otimes 1 ) (1 \otimes A) (e' \otimes 1) c = 1 \, (1 \otimes A) \, 1
= A.
\end{array}
\end{displaymath}
The commutative law follows since
\begin{eqnarray}
 A \bullet B &= m \big(  A \otimes  B\big)c &= m\cdot tw\cdot (A \otimes B ) \cdot c\nonumber\\
&=m\cdot(B \otimes A ) \cdot tw\cdot \ c   &= m (B \otimes A) c=  B
\bullet A . \nonumber
\end{eqnarray}
\end{proof}

\noindent It is useful to visualize the operation:

\centerline{{\tt\setlength{\unitlength}{0.60pt}
\begin{picture}(350,60)
\thinlines
\put(26,7){$A\bullet B$}
\put(275,-4){$B$}
\put(275,20){$A$}
\put(350,9){$m$}
\put(125,5){$=$}
\put(200,8){$c$}
\put(-28,3){$X$}
\put(103,3){$Y$}
\path(-10,10)(20,10)(20,30)(70,30)(70,-10)(20,-10)(20,10)
\path(70,10)(100,10)
%
\path(220,-10)(220,35)(200,35)(180,13)(200,-10)(220,-10)
\path(180,13)(150,13)
\path(220,0)(260,0)
\path(260,0)(260,-10)(300,-10)(300,0)(300,10)(260,10)(260,0)
\path(220,25)(260,25)(260,35)(300,35)(300,15)(260,15)(260,25)
\path(300,25)(340,25)(340,-10)(360,-10)(380,15)(360,35)(340,35)(340,15)
\path(340,0)(300,0)
\path(380,15)(400,15)
\end{picture}}}

\section{\TLTS\  and its algebraic structure}

\subsection{Reflexive graphs}

A \emph{graph} $X$ consists of a set $X_0$ of vertices of $X$, a set
$X_1$ of edges and two functions $d_0,d_1:X_1 \to X_0$ (domain and
codomain functions). A \emph{reflexive graph} $X$ is
a graph with a function $\varepsilon  : X_0 \to X_1 $ such that $d_0 \varepsilon  = d_1
\varepsilon $. For every $x \in X_0$, $\varepsilon  _x$ is the \emph{reflexive edge} of
$x$.
 
Let $X$ and $Y$ be two reflexive graphs. A \emph{morphism of reflexive graphs}
$\phi:X \to Y$ is a graph morphism such that $\phi( \varepsilon _x)=\varepsilon _{\phi(x)} $. 

\noindent The \emph{product} $X \times Y$ is the reflexive graph such that
\[(X \times Y)_0 := X_0 \times Y_0, \quad  (X \times Y)_1 := X_1
\times Y_1, \quad \varepsilon  _{X \times Y}:= \varepsilon  _ X \times\ \varepsilon  _Y. \]

\noindent We denote by $RGrph$ the category of reflexive graphs and
morphisms between them.

\subsection{The category of $\TLTS$}

The name $\TLTS$ comes from the fact that the arrows of $\TLTS$ are  \textquotedblleft two-sided labelled transition systems". 

\begin{definition} Given two sets  $X$, $Y$ both containing the symbol $\varepsilon,$ a \emph{two-sided transition system} $A$  \emph{labelled in } $X$ \emph{and} $Y$ consists of a set $A_0$ (of \emph{states}), and a subset  $A_{1}$  of $A_0\times X\times Y\times A_0$ (of \emph{transitions}) containing $(a,\varepsilon,\varepsilon,a)$ for each $a\in A_0$. 
\end{definition}
It is convenient sometimes to write the transition $(a,x,y,a')$ as $(a,x/y,a')$ or even $\xymatrix{a \ar[r]^{x/y} & aì'}.
$
\medskip

\noindent \emph{Strictly speaking we regard two transition systems $A$, $B$ with labels in $X$ and $Y$ as the same if $A_0$ is bijective with $B_0$ and the bijection respects edges and their labelling}.

\bigskip We now define the category $\TLTS$.

\begin{definition}The \emph{objects} of $\TLTS$  are sets containing the symbol $\varepsilon$, which we may think of as alphabets. Given objects $X$, $Y$ an \emph{arrow} $A$ from $X$ to $Y$   is a two-sided transition system labelled in $X$ and $Y$. The \emph{composition} $A\cdot B\ $ of $A:X\to Y$ and $B:Y\to Z$ is defined to be
$$(A\cdot B )_{0}=A_{0}\times B_0,$$
$$(A\cdot B)_1=\{(a,b,x/z,a',b');\exists y\in Y {\rm such\ that\ } (a,x/y,a')\in A_0,(b,y/z,b')\in B_{0}\}.$$
The identity arrow of $X$ has one state $*$ and transitions $\{(*,x/x,*);x\in X\}. $  
\end{definition}
The category $\TLTS$ bears a straightforward relationship with $Span(RGraph)$ - it is in fact a full subcategory of a quotient of $Span(RGraph)$, analogous to the fact that the category of relations is a quotient of $Span(Sets)$.  However we have preferred here to define $\TLTS$ explicitly.
To see how an arrow of $\TLTS$ may be considered a span of reflexive graphs one must first regard the objects as one vertex graphs, the alphabet being the set of edges, including $\varepsilon$ as the reflexive edge. Further given an  arrow $A:X\to Y$ in $\TLTS$ the two sets $A_1$ and $A_0$ form the arrows and edges of a graph; the two functions $d_0, d_1$ are defined by $d_0(a,x/y,b)=a,\  d_1(a,x/y,b)=b. $   Finally the arrow $A:X\to Y$ of $\TLTS$ yields a span of reflexive graphs
\begin{equation}\label{span1}
\xymatrix{X & A \ar[l]_{\delta_0 ^ A} \ar[r]^{\delta_1 ^ A} & Y }
\end{equation}
defined by $\delta_0(a,x/y,b)=x,\ \delta_1(a,x/y,b)=y$.
Composition in $\TLTS$ is composition in $Span(RGraph)$, followed by the reflection of general spans into spans jointly monic on arcs.
\bigskip

\noindent{\bf Examples} To see examples of two-sided labelled transition systems and their use in modelling concurrent systems, we refer to papers of the authors, beginning with \cite{KSW97b}.

\subsection{Relations} We will now see that the category $Rel_*$ of relations between pointed sets is a subcategory of $\TLTS$. Given a relation $\rho$ between two pointed sets $X$ and $Y$ (the points both denoted $\varepsilon$) with the property that $\varepsilon\rho\varepsilon$,  we obtain an arrow $\bar \rho:X\to Y$ of $\TLTS$ as follows: $\bar\rho_0=\{*\}$, $\bar\rho_1=\{(*,x/y,*);x\rho y\}. $ It is immediately clear that composition of relations in $\TLTS$ agrees with the usual composition of relations. 

\medskip

\subsection{The well-supported compact closed  structure of $\TLTS$}

 Since $RGrph$ has finite limits, $RGrph$ is a monoidal
category. The tensor product $\otimes$ of reflexive graphs is their
product. Each object $X$ of $RGrph$ has a structure of commutative
monoid in $Span(RGrph)$. In fact the spans
\begin{displaymath}
\nabla:= (\Delta_X,1_X): \; X \otimes X \to X \qquad e:=(!,1): \; I
\to X
\end{displaymath}
satisfy the axioms of the definition of monoid object, and the
multiplication $\nabla$ is compatible with the twist map.

Each reflexive graph $X$ has also a structure of commutative
comonoid. The comultiplication is \[\Delta:= (1, \Delta_X): \; X \to
X \otimes X\]

There is a symmetric monoidal structure on spans. Given two spans
$A:X \to Y$, $B:Z \to W$ the \emph{tensor} of $A$ and $B$ is defined
as
\[A \otimes B := (\delta_0 ^A \times
\delta _0 ^B,\delta_1 ^A \times \delta _1 ^B):\; X \otimes Z \to Y
\otimes W\] Given two objects $X,Y$ there is a \emph{twist} span
\begin{displaymath}
\xymatrix{X \times Y & X \times Y \ar[l]_{1} \ar[r]^{tw} & Y \times
X}
\end{displaymath}
where $tw$ is the twist map in $RGrph$.

\begin{definition}
A \emph{wscc category} is a symmetric monoidal category with for each object $X$ the structure of commutative monoid and comonoid satisfy the following axioms:
\begin{enumerate}[i)]
\item Frobenius axiom: $(\nabla \otimes 1)(1 \otimes \Delta)=\Delta \nabla $
\medskip
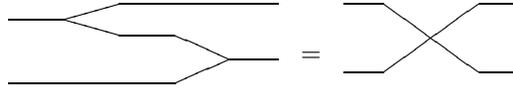
\begin{displaymath}
\begin{array}{ccccc}
\xymatrix@C=15pt@R=3pt{&& *=0{} \ar@{-}[rrr] &&& *=0{}\\
*=0{} \ar@{-}[r] & *=0{} \ar@{-}[ur] \ar@{-}[dr]\\
&& *=0{} \ar@{-}[r] & *=0{}\\
&&&& *=0{} \ar@{-}[ul] \ar@{-}[dl] & *=0{} \ar@{-}[l]\\
*=0{} \ar@{-}[rrr] &&& *=0{}  }
&& \xymatrix@C=10pt@R=3pt{\\\\=} &&
\xymatrix@C=15pt@R=20pt{*{} \ar@{-}[r] & *{} \ar@{-}[drr] && *{} \ar@{-}[r] & *{}\\
*{} \ar@{-}[r] & *{} \ar@{-}[urr] && *{} \ar@{-}[r] & *{}}
\end{array}
\end{displaymath}
\item Separable axiom: \medskip $\nabla \Delta=1$
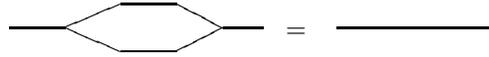
\begin{displaymath}
\begin{array}{ccccc}
\xymatrix@C=15pt@R=3pt{&& *=0{} \ar@{-}[r] & *=0{}\\
*=0{} \ar@{-}[r] & *=0{} \ar@{-}[ur] \ar@{-}[dr] &&& *=0{}\ar@{-}[r] \ar@{-}[ul] \ar@{-}[dl]& *=0{}\\
&& *=0{} \ar@{-}[r] & *=0{} }
&& \xymatrix@C=10pt@R=3pt{\\=} &&
\xymatrix@C=15pt@R=3pt{\\*{} \ar@{-}[rrr] &&& *{} }
\end{array}
\end{displaymath}
\end{enumerate}
\end{definition}

Derived operations are:
\begin{itemize}
\item the \emph{projection} $\xymatrix{X \otimes Y \ar[r]^{X \otimes \text{\textexclamdown}} & X \otimes I \ar[r]^-{1} & X}$,
\item the \emph{opposite projection} $\xymatrix{I \ar[r]^1 & X \otimes I \ar[r]^{X \otimes !} & X \otimes Y }$,
\item the \emph{unit} $\eta _X: \xymatrix{I \ar[r]^{!} & X \ar[r]^-{\Delta} & X \otimes X}$,
\item the \emph{counit} $\epsilon_X: \xymatrix{X \otimes X \ar[r]^{\nabla} & X \ar[r]^{\text{\textexclamdown}} & I}$.
\\
\end{itemize}

\subsection{Monoids in $Rel_*$}What is a commutative monoid in $\TLTS$ in which the structure arrows of the monoid are relations?
It is easy to verify the following proposition:

\begin{proposition}

A commutative monoid structure in $\TLTS$ on object $X,$ for which all the structure maps are pointed relations, amounts to (i) a subset $e$ of $X$ containing $\varepsilon$,  (ii) a function $m:X \times\ X\to\wp(X)$ satisfying,
for all $x,y,z$ in $X$, $m(e,x)=\{x\}$, $m(x,e)=\{x\}$, $m(x,y)=m(y,x) $ and   $m(x,y),z))=m(m(x,m(y,z))$. Notice that the conditions involve extending the definition of $m$ in the obvious way to subsets of $X$. 
\end{proposition}

\section{Parallel composition in Process Algebras based on broadcast}

It is undoubtedly true that one of the most common ways of connecting components is by broadcast; that is, each component may communicate which any other directly. The geometry is \emph{not} 
$$
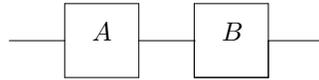
\centerline{\tt\setlength{\unitlength}{3.5pt}
\begin{picture}(42,16)\thinlines
\put(13,8){$A$}
\put(27,8){$B$}
\drawpath{4.0}{8.0}{10.0}{8.0}
\drawpath{10.0}{8.0}{10.0}{12.0}
\drawpath{10.0}{12.0}{18.0}{12.0}
\drawpath{18.0}{12.0}{18.0}{4.0}
\drawpath{18.0}{4.0}{10.0}{4.0}
\drawpath{10.0}{4.0}{10.0}{8.0}
\drawpath{18.0}{8.0}{24.0}{8.0}
\drawpath{24.0}{8.0}{24.0}{12.0}
\drawpath{24.0}{12.0}{32.0}{12.0}
\drawpath{32.0}{12.0}{32.0}{4.0}
\drawpath{32.0}{4.0}{24.0}{4.0}
\drawpath{24.0}{4.0}{24.0}{8.0}
\drawpath{24.0}{8.0}{24.0}{8.0}
\drawpath{24.0}{8.0}{24.0}{4.0}
\drawpath{24.0}{4.0}{32.0}{4.0}
\drawpath{32.0}{4.0}{32.0}{8.0}
\drawpath{32.0}{8.0}{38.0}{8.0}
\end{picture}} $$but something like
$$
\tt\setlength{\unitlength}{6.0pt}
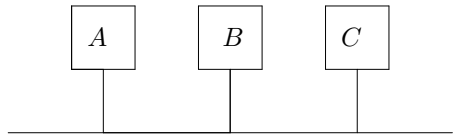
\begin{picture}(34,18)
\thinlines
\put(9,13.5){$A$}
\put(17.5,13.5){$B$}
\put(25,13.5){$C$}
\drawpath{4.0}{8.0}{32.0}{8.0}
\drawpath{8.0}{16.0}{8.0}{12.0}
\drawpath{8.0}{12.0}{12.0}{12.0}
\drawpath{12.0}{12.0}{12.0}{16.0}
\drawpath{12.0}{16.0}{8.0}{16.0}
\drawpath{8.0}{16.0}{8.0}{12.0}
\drawpath{8.0}{12.0}{10.0}{12.0}
\drawpath{10.0}{12.0}{10.0}{8.0}
\drawpath{10.0}{8.0}{18.0}{8.0}
\drawpath{18.0}{8.0}{18.0}{12.0}
\drawpath{18.0}{12.0}{20.0}{12.0}
\drawpath{20.0}{12.0}{20.0}{16.0}
\drawpath{20.0}{16.0}{16.0}{16.0}
\drawpath{16.0}{16.0}{16.0}{12.0}
\drawpath{16.0}{12.0}{18.0}{12.0}
\drawpath{18.0}{12.0}{18.0}{8.0}
\drawpath{26.0}{12.0}{28.0}{12.0}
\drawpath{28.0}{12.0}{28.0}{16.0}
\drawpath{28.0}{16.0}{24.0}{16.0}
\drawpath{24.0}{16.0}{24.0}{12.0}
\drawpath{24.0}{12.0}{26.0}{12.0}
\drawpath{26.0}{12.0}{26.0}{8.0}
\end{picture}
$$
 We might call the bottom line here a \emph{bus}. The components can talk directly to each other through this medium, though naturally in interleaving.  
 We maintain however that the first geometry above is the canonical one, while the second is a special derived operation,  namely, in terms of the wscc operations of $\TLTS, $ 

$$\nabla\;(C\otimes 1)\;\nabla\;(B\otimes 1)\;A=\nabla\;(C\otimes(\nabla\;(B\otimes A)))   $$
or geometrically
$$
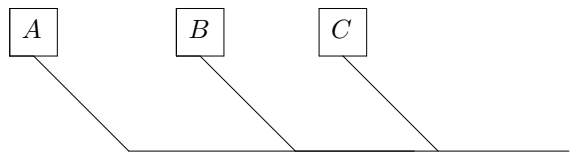
\centerline{\tt\setlength{\unitlength}{4.5pt}
\begin{picture}(64,20)
\thinlines
\put(9,13.5){$A$}
\put(23,13.5){$B$}
\put(35,13.5){$C$}
\drawpath{18.0}{4.0}{55.0}{4.0}
\drawpath{8.0}{16.0}{8.0}{12.0}
\drawpath{8.0}{12.0}{12.0}{12.0}
\drawpath{12.0}{12.0}{12.0}{16.0}
\drawpath{12.0}{16.0}{8.0}{16.0}
\drawpath{8.0}{16.0}{8.0}{12.0}
\drawpath{8.0}{12.0}{10.0}{12.0}
\drawpath{10.0}{12.0}{18.0}{4.0}
\drawpath{22.0}{16.0}{22.0}{12.0}
\drawpath{22.0}{12.0}{26.0}{12.0}
\drawpath{26.0}{12.0}{26.0}{16.0}
\drawpath{26.0}{16.0}{22.0}{16.0}
\drawpath{22.0}{16.0}{22.0}{12.0}
\drawpath{22.0}{12.0}{24.0}{12.0}
\drawpath{24.0}{12.0}{32.0}{4.0}
\drawpath{32.0}{4.0}{42.0}{4.0}
\drawpath{44.0}{4.0}{36.0}{12.0}
\drawpath{36.0}{12.0}{34.0}{12.0}
\drawpath{34.0}{12.0}{34.0}{16.0}
\drawpath{34.0}{16.0}{38.0}{16.0}
\drawpath{38.0}{16.0}{38.0}{12.0}
\drawpath{38.0}{12.0}{36.0}{12.0}
\end{picture}
} $$
This expression certainly acts like pure broadcast: in a transition of the whole systems the transitions of each component must have the same label on the "bus". Note that we could have as easily, and perhaps more naturally, used $\Delta$ rather than $\nabla$; however the comparison with synchronization algebras is simpler using $\nabla$.
 
\medskip\  What we will describe next is a modification of pure broadcast in which there is a protocol between the processes and the bus.

\medskip\ 
\subsection{Classical labelled transition systems and synchronization algebras}

Let $\mathcal{L}$ be an alphabet which does not include the symbols
$0$ and $\varepsilon$. Define \[\mathcal{L}_{\epsilon}:= \mathcal{L} \cup \{\varepsilon\}
\qquad \mathcal{L}_{\varepsilon ,0}:= \mathcal{L} \cup \{\varepsilon,0 \}\]

\begin{definition}\cite{WN} A \emph{synchronization algebra} on $\mathcal{L}$ is a binary,
commutative and associative operation $\diamond$ on $\mathcal{L}
_{\varepsilon,0 }$ such that for all $\alpha,\beta \in \mathcal{L}_ {\varepsilon,0 }$

\begin{enumerate}
\item[(i)] $\alpha \diamond\ 0=0$,
\item[(ii)] $\alpha \diamond \beta = \varepsilon \;$ if and only if $\; \alpha=\beta=\varepsilon.$
\end{enumerate}
\end{definition}
The idea is that the element $0$ denotes those synchronizations which are not
allowed,  $\varepsilon$ allows asynchrony, and $\alpha\diamond \beta$ is the resulting signal on the bus when messages $\alpha$ and $\beta$ are passed from components.
\par\bigskip
\noindent{\bf Remark\par} If we regard $\alpha\diamond \beta=0$ as meaning that  $\alpha\diamond \beta$ is undefined, then we may think of the operation of  a synchronization algebra as being a partial function. Property (i) assures us that no information is lost. Then clearly by Proposition 2 a synchronization algebra is a commutative semigroup object in $Rel_*$, and hence in $\TLTS$.
In fact a synchronization algebra on alphabet $\mathcal{L}$ is the same thing as a commutative semigroup object in $Rel_*$ on object $\mathcal{L}_{\varepsilon}$ whose multiplication is a partial function, and which satisfies the \emph{additional} property that $\alpha\diamond \beta=\varepsilon$ implies $\alpha=\beta=\varepsilon$.  
\bigskip

As usual, a \emph{transition system} $A$ labelled in $\cal{L}$ consists of  a set $S$ of \emph{states},
and a \emph{transition relation} $T \subseteq S \times \mathcal{L} \times S$.

\medskip
\begin{definition}
\cite{WN} Let $A=(S_A, \mathcal{L},T_A)$, $B=(S_B, \mathcal{L},T_B)$ be
transition systems on the same alphabet $\mathcal{L}$. Given a
synchronization algebra $\diamond$ on $\mathcal{L}$, the
\emph{parallel composition} of $A$ and $B$ is the transition system
$
A \parallel B := ( \, S_{\parallel}, \; \mathcal{L}, \;
T_{\parallel} \; )
$
where:
\begin{enumerate}
\item[] $S_{\parallel}:=S_A \times S_B$,
\item[] $T_{\parallel}:= \{ \big( (a,b),\lambda,(a',b') \big) \; | \; \lambda=\alpha
\diamond \beta  \neq 0, \; (x \stackrel{\alpha}{\rightarrow} x', y
\stackrel{\beta}{\rightarrow} y'  ) \in T_{\times}\}$,
\end{enumerate}
where 
$T_{ \times }= T_A \times T_B+ T_A \times \{b \stackrel{\varepsilon}{\rightarrow} b \;|\; b \in
S_B\} + \{a \stackrel{\varepsilon}{\rightarrow} a \;|\; a \in S_A\} \times
T_B$.
\end{definition}
\bigskip

It is straightforward to check the following proposition:

\begin{proposition}
Consider a synchronization algebra $\diamond$ on the alphabet $\mathcal{L}$, considered now as a commutative semigroup object in $\TLTS$. Let $\bullet$ be the commutative semigroup operation induced on $Hom(I,\mathcal{L}_{\varepsilon})$, as in Proposition 1 ($I$ has a trivial cocommutative comonoid structure). Then

$$A\parallel B=A\bullet B. $$

\end{proposition}
In the geometry of expressions in the wscc category $\TLTS$, $A\parallel B$ is:
$$
\centerline{\tt\setlength{\unitlength}{3pt} 
\begin{picture}(36,24)
\thinlines
\put(5,17.0){$A$}
\put(5,5.0){$B$}
\put(25,7.5){$m$}
\drawpath{4.0}{20.0}{4.0}{16.0}
\drawpath{4.0}{16.0}{8.0}{16.0}
\drawpath{8.0}{16.0}{8.0}{20.0}
\drawpath{8.0}{20.0}{4.0}{20.0}
\drawpath{4.0}{20.0}{8.0}{20.0}
\drawpath{8.0}{20.0}{8.0}{18.0}
\drawpath{8.0}{18.0}{16.0}{18.0}
\drawpath{16.0}{18.0}{24.0}{14.0}
\drawpath{24.0}{14.0}{24.0}{10.0}
\drawpath{24.0}{10.0}{16.0}{6.0}
\drawpath{16.0}{6.0}{8.0}{6.0}
\drawpath{18.0}{10.0}{18.0}{10.0}
\drawpath{8.0}{6.0}{8.0}{8.0}
\drawpath{8.0}{8.0}{4.0}{8.0}
\drawpath{4.0}{8.0}{4.0}{4.0}
\drawpath{4.0}{4.0}{8.0}{4.0}
\drawpath{8.0}{4.0}{8.0}{6.0}
\drawpath{24.0}{14.0}{26.0}{12.0}
\drawpath{26.0}{12.0}{24.0}{10.0}
\drawpath{24.0}{10.0}{26.0}{12.0}
\drawpath{26.0}{12.0}{32.0}{12.0}
\end{picture}
}
$$
It is also straightforward to see that $A\parallel B\parallel C$ has the following geometry, analogous to the example of pure broadcast above:

$$
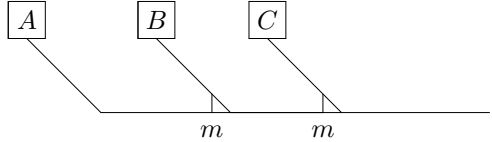
\centerline{\tt\setlength{\unitlength}{3.5pt}
\begin{picture}(64,24)
\thinlines
\drawpath{18.0}{8.0}{60.0}{8.0}
\drawpath{8.0}{20.0}{8.0}{16.0}
\drawpath{8.0}{16.0}{12.0}{16.0}
\drawpath{12.0}{16.0}{12.0}{20.0}
\drawpath{12.0}{20.0}{8.0}{20.0}
\drawpath{8.0}{20.0}{8.0}{16.0}
\drawpath{8.0}{16.0}{10.0}{16.0}
\drawpath{10.0}{16.0}{18.0}{8.0}
\drawpath{22.0}{20.0}{22.0}{16.0}
\drawpath{22.0}{16.0}{26.0}{16.0}
\drawpath{26.0}{16.0}{26.0}{20.0}
\drawpath{26.0}{20.0}{22.0}{20.0}
\drawpath{22.0}{20.0}{22.0}{16.0}
\drawpath{22.0}{16.0}{24.0}{16.0}
\drawpath{24.0}{16.0}{32.0}{8.0}
\drawpath{32.0}{8.0}{42.0}{8.0}
\drawpath{44.0}{8.0}{36.0}{16.0}
\drawpath{36.0}{16.0}{34.0}{16.0}
\drawpath{34.0}{16.0}{34.0}{20.0}
\drawpath{34.0}{20.0}{38.0}{20.0}
\drawpath{38.0}{20.0}{38.0}{16.0}
\drawpath{38.0}{16.0}{36.0}{16.0}
\drawpath{30.0}{10.0}{30.0}{8.0}
\drawpath{42.0}{10.0}{42.0}{8.0}
\drawcenteredtext{10.0}{18.0}{$A$}
\drawcenteredtext{24.0}{18.0}{$B$}
\drawcenteredtext{36.0}{18.0}{$C$}
\drawcenteredtext{30.0}{6.0}{$m$}
\drawcenteredtext{42.0}{6.0}{$m$}
\end{picture}
}
$$
The expression here is $m(C\otimes 1)m(B\otimes 1)A.\medskip $
\par
Notice that $A\parallel B\parallel C$=$(C\parallel A) \parallel B$ and hence $C$ may communicate directly with $A$, and also that the order on the bus is irrelevant.

\subsection{Examples}

\subsubsection{Pure Broadcast} We have already discussed this case which arises from the comonoid structure of objects which is part of the wscc structure of $\TLTS$. The comultiplication is the arrow 
\[\nabla : \mathcal{L}_{\varepsilon}\otimes \mathcal{L}_{\varepsilon} \to \mathcal{L}_{\varepsilon}\]
which is actually the partial function $\mathcal{L}_{\varepsilon}{\times} \mathcal{L}_{\varepsilon}\stackrel{\Delta}{\longleftarrow} \mathcal{L}_{\varepsilon}\stackrel{1}{\longrightarrow} \mathcal{L}_{\varepsilon}$.
Notice however that the unit of the monoid structure is not a partial function, which means that the synchronization algebra is only a semigroup. It is our view that the extension of the notion of synchronization algebra to monoids in $Rel_*$ is important. 

\subsubsection{CCS}

The  alphabet $\mathcal{L}$ 
contains a special  letter $\tau$ and to each other letter $\alpha \in \mathcal{L}$,
$\alpha \neq \tau$, there exists the complementary label $\bar{\alpha} \in
\mathcal{L}$. The multiplication $\diamond$ on $\mathcal{L}_{\varepsilon}$ is the partial map defined as follows:
\begin{enumerate}
\item[(i)] $\alpha\diamond\varepsilon=\varepsilon\diamond\alpha=\alpha$ for all $\alpha$ (including $\tau$),
\item[(ii)] $\alpha\neq\tau$, implies that $\alpha\diamond\bar{\alpha}=\tau=\bar{\alpha}\diamond{\alpha}$,
\item[(iii)] on all other pairs $\diamond$ is undefined.
\end{enumerate}
This multiplication does have an identity element, namely the element $\varepsilon$.
\subsubsection{Non-reflexive graphs and synchronization}  An important special case of broadcast is the clock signal in synchronous machines. The clock has one vertex and one non-reflexive edge, the clock signal. 
$$
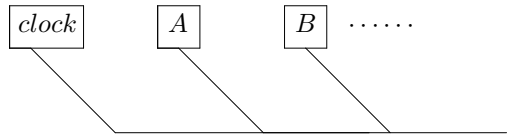
\centerline{\tt\setlength{\unitlength}{4.pt}
\begin{picture}(64,20)
\thinlines
\put(8.8,13.5){$clock$}
\put(23,13.5){$A$}
\put(35,13.5){$B$}
\put(40,13.5){$\cdots\cdots$}
\drawpath{18.0}{4.0}{55.0}{4.0}
\drawpath{8.0}{16.0}{8.0}{12.0}
\drawpath{8.0}{12.0}{15.0}{12.0}
\drawpath{15.0}{12.0}{15.0}{16.0}
\drawpath{15.0}{16.0}{8.0}{16.0}
\drawpath{8.0}{16.0}{8.0}{12.0}
\drawpath{8.0}{12.0}{10.0}{12.0}
\drawpath{10}{12.0}{18.0}{4.0}
\drawpath{22.0}{16.0}{22.0}{12.0}
\drawpath{22.0}{12.0}{26.0}{12.0}
\drawpath{26.0}{12.0}{26.0}{16.0}
\drawpath{26.0}{16.0}{22.0}{16.0}
\drawpath{22.0}{16.0}{22.0}{12.0}
\drawpath{22.0}{12.0}{24.0}{12.0}
\drawpath{24.0}{12.0}{32.0}{4.0}
\drawpath{32.0}{4.0}{42.0}{4.0}
\drawpath{44.0}{4.0}{36.0}{12.0}
\drawpath{36.0}{12.0}{34.0}{12.0}
\drawpath{34.0}{12.0}{34.0}{16.0}
\drawpath{34.0}{16.0}{38.0}{16.0}
\drawpath{38.0}{16.0}{38.0}{12.0}
\drawpath{38.0}{12.0}{36.0}{12.0}
\end{picture}
} $$
If each non-reflexive edge of  $A,B,C,\cdots$ is labelled by the clock signal, then this expression evaluates to the product, in non-reflexive Graphs, of the graphs consisting of the non-reflexive edges of $A$,$B$,$C,\cdots$, .


\end{document}